\documentclass[11pt,reqno]{amsart}
\usepackage{amssymb,mathrsfs,graphicx,enumerate}
\usepackage{amsmath,amsfonts,amssymb,amscd,amsthm,bbm}

\usepackage[retainorgcmds]{IEEEtrantools}
\usepackage{colortbl}
\usepackage{cite}
\usepackage{subcaption}
\usepackage{nccmath}
\usepackage{bm}
\usepackage{upgreek}
\title[TCS with unit speed and collision avoidance]{Thermodynamic Cucker-Smale ensemble with unit speed and its sufficient framework for collision avoidance}

\author[Ahn]{Hyunjin Ahn}
\address[Hyunjin Ahn]{\newline Department of Mathematics, Myongji University, Gyeonggi-do 17058, Republic of Korea}
\email{yagamelaito@snu.ac.kr}

\author[Byeon]{Junhyeok Byeon$^\dagger$}
\address[Junhyeok Byeon]{\newline Research Institute of Basic Sciences, Seoul National University, Seoul 08826, Republic of Korea}
\email{giugi2486@snu.ac.kr}

\author[Ha]{Seung-Yeal Ha}
\address[Seung-Yeal Ha]{\newline Department of Mathematical Sciences and Research Institute of Mathematics, Seoul National University, Seoul 08826, Republic of Korea}
\email{syha@snu.ac.kr}

\newtheorem{theorem}{Theorem}[section]
\newtheorem{lemma}{Lemma}[section]

\newtheorem{proposition}{Proposition}[section]
\newtheorem{remark}{Remark}[section]
\newtheorem{example}{Example}[section]

\newtheorem{definition}{Definition}[section]

\newcommand{\bbr}{\mathbb R}
\newcommand{\bbc}{\mathbb C}

\newcommand{\vast}{\bBigg@{4}}
\newcommand{\Vast}{\bBigg@{5}}

\begin{document}

\date{\today}


\thanks{\textbf{Acknowledgment.} 
The work of H. Ahn was supported by the National Research Foundation of Korea(NRF) grant funded by the Korea government(MSIT) (2022R1C12007321), {the work of J. Byeon was supported by the National Research Foundation of Korea(NRF) grant funded by the Korean government (MEST) (N0.2019R1A6A1A10073437)} and the work of S.-Y. Ha was partially supported by the National Research Foundation of Korea Grant (NRF-2020R1A2C3A01003881). \newline
$\dagger$: Corresponding author}
\begin{abstract}	
We investigate a Cucker-Smale-type flocking model for multi-agent systems that move with constant speed. The model incorporates both kinematic observables and internal energy (temperatures) in the agents' interactions. Traditionally, collision avoidance in the absence of speed limitation is achieved by introducing singularities into the communication rule. However, when a unit speed constraint is applied, the mechanism of collision avoidance can differ, and the singularity may not necessarily prevent collisions. In this paper, we propose a framework that generates collision avoidance, asymptotic flocking, thermal equilibrium, and strict spacing between agents, subject to sufficient conditions expressed by the initial condition, system parameters, and degree of singularity.
\end{abstract}

\maketitle \centerline{\date}


\section{Introduction} \label{sec:1}
\setcounter{equation}{0}

Emergent flocking dynamics are commonly observed in interacting many-body systems in nature and human society, such as the synchronization of fireflies and pacemaker cells \cite{B-B, Er, Wi2}, aggregation of bacteria \cite{T-B}, flocking of birds \cite{C-S}, and swarming of fish \cite{D-M1, T-T}. These fields are briefly introduced in relevant literature \cite{A-B, A-B-F, C-H-L, F-T-H, O2, P-R-K, Str, VZ, Wi1}. In this paper, we focus on flocking behavior, where particles move with a common velocity using simple rules in a limited environment.

Several dynamical systems have been proposed to study emergent behaviors in flocking systems, including the Kuramoto model, Cucker-Smale, and Winfree model, following Vicsek et al.'s work on the flocking model in \cite{V-C-B-C-S}. Since \cite{C-S}, researchers have investigated various aspects of the Cucker-Smale model, which is constructed using Newton-like second-order models for position-velocity for the mean-field limit \cite{A-H-K, A-H-K-S-S, HKMRZ, H-L, H-K-Z}, kinetic models \cite{C-F-R-T, H-T}, hydrodynamic descriptions \cite{F-K, H-K-K, K-M-T1}, particle analysis \cite{C-D-P, C-F-R-T, C-H, C-H2, C-H-H-J-K, C-H-H-J-K2, C-H-L, C-K-P-P, C-L, C-H2, C-C-H-K-K}, temperature field \cite{H-K-Rug2, H-R}, relativistic setting \cite{A-H-K, A-H-K-S-S, A-H-K-S, B-H-K, H-K-R}, collision avoidance \cite{C-C-H-K-K, C-D1, C-D2, C-D3, C-K-P-P, M-P, P, P-K-H}, and others.

Previous studies have mainly focused on the Cucker-Smale model with constant speed, and have investigated aspects such as basic particle analysis \cite{C-H2}, bi-cluster flocking \cite{C-H-H-J-K}, multi-cluster flocking and necessary conditions for mono-cluster flocking \cite{H-K-Z2}, and time-delay effects \cite{C-H3}. However, these studies were originally motivated by the Cucker-Smale model, and the author \cite{A2} extended it to the unit-speed Cucker-Smale model with a temperature field. To motivate our study, we briefly introduce the dynamic system of the thermodynamic Cucker-Smale model, which is a second-order system for \emph{position-velocity-temperature} $(x_i, v_i, T_i)$, as given in \cite{H-K-Rug2}:

\begin{equation}
\begin{cases} \label{TCS}
\displaystyle \frac{d{x}_i}{dt} =v_i,\quad t>0,\quad  i=1,\cdots,N,\\
\displaystyle \frac{dv_i}{dt}=\displaystyle\frac{\kappa_1}{N}\sum_{j=1}^{N}\phi(\|x_i-x_j\|)\left(\frac{v_j}{T_j}-\frac{v_i}{T_i}\right),\\
\displaystyle \frac{d}{dt}\left({T}_i+\frac{1}{2}{\|v_i\|^2}\right)=\displaystyle\frac{\kappa_2}{N}\sum_{j=1}^{N}\zeta(\|x_i-x_j\|)\left(\frac{1}{T_i}-\frac{1}{T_j}\right),\\
\displaystyle (x_i(0),v_i(0),T_i(0))=(x_i^0,v_i^0,T_i^0)\in \mathbb{R}^{2d}\times \mathbb{R}_{>0},
\end{cases}
\end{equation}
where $N$ is the number of particles, and $\kappa_1$ and $\kappa_2$ are strictly positive coupling strengths. The communication weights $\phi$ and $\zeta$ are non-negative, bounded, locally Lipschitz continuous and monotone decreasing functions mapping from $\mathbb{R}+$ to $\mathbb{R}_+$.

To modify the velocity coupling term in $\eqref{TCS}_2$ and ensure unit speed of the particles, the author in \cite{A2} applied a derivation idea from the unit-speed Cucker-Smale studied in \cite{C-H2} as shown below:

\[\phi(\|x_i-x_j\|)\left(\frac{v_j}{T_j}-\frac{v_i}{T_i}\right) \rightarrow \phi(\|x_i-x_j\|)\left(\frac{v_j}{T_j}-\frac{\langle v_j,v_i\rangle v_i}{T_j\|v_i\|^2}\right).\]
\indent As a result, the thermodynamic Cucker-Smale model with unit speed is proposed as follows,

\begin{equation}
\begin{cases} \label{TCSUS}
\displaystyle \frac{d{x}_i}{dt} =v_i,\quad t>0,\quad  i=1,\cdots,N,\\
\displaystyle \frac{dv_i}{dt}=\displaystyle\frac{\kappa_1}{N}\sum_{j=1}^{N}\phi(\|x_i-x_j\|)\left(\frac{v_j}{T_j}-\frac{\langle v_j,v_i\rangle v_i}{T_j\|v_i\|^2}\right),\\
\displaystyle \frac{dT_i}{dt}=\displaystyle\frac{\kappa_2}{N}\sum_{j=1}^{N}\zeta(\|x_i-x_j\|)\left(\frac{1}{T_i}-\frac{1}{T_j}\right),\\
\displaystyle (x_i(0),v_i(0),T_i(0))=(x_i^0,v_i^0,T_i^0)\in \mathbb{R}^{d}\times\mathbb{S}^{d-1}\times \mathbb{R}_{>0},
\end{cases}
\end{equation}where $N$, $\kappa_1$, $\kappa_2$ and the communication weight $\phi$, $\zeta$ are defined as the above \eqref{TCS} and $\mathbb{S}^d$ is an unit $d$-sphere. Specifically, we have

\begin{align}\label{A-1}
\begin{aligned}
&0\leq\phi(r)\leq \phi(0)=1,~ (\phi(r_1)-\phi(r_2))(r_1-r_2)\leq 0,~ \forall r,r_1,r_2\geq 0,~ \phi(\cdot)\in \text{Lip}^{\text{loc}}(\mathbb{R}_+;\mathbb{R}_+),\\
&0\leq\zeta(r)\leq \zeta(0)=1,~ (\zeta(r_1)-\zeta(r_2))(r_1-r_2)\leq 0,~\forall r,r_1,r_2\geq 0,~ \zeta(\cdot)\in \text{Lip}^{\text{loc}}(\mathbb{R}_+;\mathbb{R}_+),\\
&\mathbb{S}^{d-1}:=\left\{x:=(x^1,\cdots,x^d)\bigg|\sum_{i=1}^d |x^i|^2=1,\right\}~\text{where $x^i$ is $i$-th component of vector $x\in\mathbb{R}^d$.}
\end{aligned}
\end{align} In particular, we note that $\eqref{TCS}_3$ is equivalent to the $\eqref{TCSUS}_3$ by the following lemma.

\begin{lemma}\label{L1.1}\emph{\cite{A2}}
	Let $\{(x_i,v_i,T_i)\}_{i=1}^N$ be a solution to the system \eqref{TCSUS}. Then, one has
	\[\|v_i(t)\|=1,\quad\forall t\in [0,\infty),\hspace{0.2cm}\forall i=1,\cdots, N.\]
\end{lemma}
\noindent Therefore, the model \eqref{TCSUS} could be simplified using Lemma \ref{L1.1} as follows. 

\begin{equation}
\begin{cases} \label{TCSUSS}
\displaystyle \frac{d{x}_i}{dt} =v_i,\quad t>0,\quad  i=1,\cdots,N,\\
\displaystyle \frac{dv_i}{dt}=\displaystyle\frac{\kappa_1}{N}\sum_{j=1}^{N}\phi(\|x_i-x_j\|)\left(\frac{v_j-\langle v_j,v_i\rangle v_i}{T_j}\right),\\
\displaystyle \frac{dT_i}{dt}=\displaystyle\frac{\kappa_2}{N}\sum_{j=1}^{N}\zeta(\|x_i-x_j\|)\left(\frac{1}{T_i}-\frac{1}{T_j}\right),\\
\displaystyle (x_i(0),v_i(0),T_i(0))=(x_i^0,v_i^0,T_i^0)\in \mathbb{R}^{d}\times\mathbb{S}^{d-1}\times \mathbb{R}_{>0}.
\end{cases}
\end{equation}
\indent The thermodynamic system with unit speed constraints \eqref{TCSUSS} and its variants have been studied in the mathematical community. Notable examples include mono-cluster and bi-cluster flocking \cite{A2} and time-delay effect \cite{A1}. Throughout this paper, we adopt the simplest singular communication weight, or singular interaction kernel, denoted by $\phi$, for collision avoidance between each pair of all particles:

\[\phi(r):=\frac{1}{r^{\alpha}},\]

Note that since we are only concerned with the singularity when $r=0$, the explicit structure of $\phi$ is not essential. The avoidance of collisions between particles (or agents) is an important issue in the fields of mechanical engineering and motion control engineering. As such, it is reasonable to investigate the collision avoidance of many-body autonomous systems.\\

\begin{remark}
We assume the properties in \eqref{A-1} for the communication weight function $\zeta$. However, it will be demonstrated in the forthcoming proof that $\zeta$ can also have a singularity in the form of $\zeta(r):=\frac{1}{r^{\beta}}$, as in \cite{A}.
\end{remark}

\indent Before proceeding with our analysis of the model \eqref{TCSUSS}, we first need to recall the definitions of asymptotic flocking and thermal equilibrium.

\begin{definition}\label{D1.1}
	Let $Z=\{(x_i,v_i,T_i)\}_{i=1}^N$ be a global-in-time solution to the system \eqref{TCSUSS}.
\begin{enumerate}
\item The configuration $Z$ exhibits asymptotic flocking if the following conditions hold:
\begin{align*}
\begin{aligned}
&(i)~\text{Group formation:} \quad \sup_{0 \leq t < \infty} \max_{1\leq i,j\leq N} \|x_i(t)-x_j(t)\|<\infty,\\
&(ii)~ \text{Velocity alignment:} \quad \lim_{t \to \infty} \max_{1\leq i,j\leq N}  \|v_j(t)-v_i(t)\|=0,
\end{aligned}
\end{align*}
\item The configuration $Z$ exhibits thermal equilibrium if the following condition holds:
\[\qquad(iii)\text{Temperature equilibrium:} \quad \lim_{t \to \infty} \max_{1\leq i,j\leq N}  |T_j(t)-T_i(t)|=0.\]
\end{enumerate}
\end{definition}

\indent Our main goal in this paper is to demonstrate the global well-posedness (i.e., collision avoidance) of the system \eqref{TCSUSS} and its emergence dynamics under a sufficient framework in terms of initial data and system parameters. 

In the presence of a unit speed constraint, several mechanisms may differ from those used in the absence of speed limitation. For example, momentum is not conserved and collision may happen even if an interaction is strongly singular (see Example \ref{E2.1}). Therefore, it will be meaningful to investigate a framework that generates certain emergent dynamics. Our main goal in this paper is to demonstrate the global well-posedness (i.e., collision avoidance) of the system \eqref{TCSUSS} and its emergence dynamics under a sufficient framework in terms of initial data and system parameters. \newline

The paper is structured as follows. In Section \ref{sec:2}, we provide a brief review of previous results on the temperature field in $\eqref{TCSUSS}_3$ for the global well-posedness and present basic estimates that are crucial for Section \ref{sec:3} and Section \ref{sec:4}. In Section \ref{sec:3}, we study a sufficient framework for the global well-posedness of $\eqref{TCSUSS}$ under strongly singular interaction kernel $(\alpha\geq 1)$ and demonstrate the emergence of asymptotic flocking and temperature equilibrium under suitable initial data and systemic parameters. In Section \ref{sec:4}, we investigate the existence of collision in the two-particle system under weakly singular communication. Then, we verify the emergence dynamics of $\eqref{TCSUSS}$ including collision avoidance, asymptotic flocking, thermal equilibrium, and strict spacing between agents, where the explicit condition depends on the degree of singularity. Finally, Section \ref{sec:5} presents a brief summary of our main results and a discussion of remaining issues to be explored in future work. \\ \newline
\indent Before we describe our main results, we set the following notations for simplicity:

\begin{align*}
&\|\cdot\|=l_2~\text{norm},\quad \langle \cdot,\cdot\rangle=\text{standard inner product},\quad y^i=\text{i-th component of}~y\in\mathbb{R}^d,\\
&X:=(x_1,\cdots,x_N),~ V:=(v_1,\cdots,v_N),~ T:=(T_1,\cdots,T_N)~ \text{in the system}~ \eqref{TCSUSS},\\
& D_{Z}(t):=\max_{i,j}\|z_i(t)-z_j(t)\|\hspace{0.3cm}\text{for}\hspace{0.3cm} Z=(z_1,\cdots,z_N)\in \{X,V,T\},\hspace{0.3cm} \mathcal{A}(v):=\min_{ 1\leq i,j\leq N}\langle v_i,v_j\rangle.
\end{align*}
Note that the diameter functionals $D_X$, $D_V$ and $D_T$ are Lipschitz continuous and so these are almost everywhere differentiable. Hence, whenever we derive a differential inequality with respect to the diameter on a subset $t\in E\subset \mathbb{R}$, then it holds on a.e. $t\in E$. Since the diameters are Lipschitz function as a maximum of Lipschitz functions, it possible to estimate the diameters via Gr\"onwall's lemma. For the detailed descriptions, refer to Section \ref{sec:3} and Section \ref{sec:4}.

\section{Preliminaries} \label{sec:2}
\setcounter{equation}{0}

In this section, we provide an overview of the fundamental estimates that are essential for Section \ref{sec:3} and Section \ref{sec:4}. Additionally, we briefly review previous findings related to the sub-dynamical system $\eqref{TCSUSS}_3$ to ensure the global well-posedness of the system \eqref{TCSUSS}.

\subsection{Previous Results} \label{sec:2.1}

In this subsection, we revisit several previous facts regarding the sub-dynamical temperature system $\eqref{TCSUSS}_3$, which will be crucially used throughout the paper.

\subsubsection{Basic Properties}

In \cite{H-R}, the authors verified the conservation law of temperature sum and the entropy principle, which can be described as follows.

\begin{definition}\label{D2.1}\emph{\cite{H-R}}
	Let $\tau\in (0,\infty]$ and $(X,V,T)$ be a local-in-time solution to the singular system \eqref{TCSUSS} on $t\in (0,\tau)$. Then, the entropy is denoted by  
	\[\mathcal{S}:=\sum_{i=1}^{N}\ln(T_i).\]
\end{definition}

\begin{proposition}\emph{\cite{H-R}}\label{P2.1}
	Let $\tau\in (0,\infty]$. Suppose that $(X,V,T)$ be a local-in-time solution to the singular system \eqref{TCSUSS} on $t\in (0,\tau)$. Then, the following assertions hold.
	\begin{enumerate}
		\item (Conserved temperature sum) The total sum  $\sum_{i=1}^{N}T_i$ is conserved:
		\[\sum_{i=1}^{N}T_i(t)=\sum_{i=1}^N T_i^{0}:=NT^\infty,\quad \forall t\in [0,\tau).\]
		\item (Entropy principle) An entropy $\mathcal{S}$ is monotonically increasing:
		\[\frac{d\mathcal{S}(t)}{dt}=\frac{1}{2N}\sum_{i,j=1}^{N}\zeta(\|x_j-x_i\|)\left|\frac{1}{T_i}-\frac{1}{T_j}\right|^2\geq 0,\quad \forall t\in [0,\tau).\]
	\end{enumerate}
\end{proposition}
\indent Thus, by using the entropy principle one obtains the uniform boundedness of temperature of each particle on $t\in [0,\tau)$ in the system \eqref{TCSUSS}.

\begin{proposition}\emph{\cite{H-K-Rug2}}\label{P2.2}~\emph{(Monotonocity of max-min temperatures)}
	Let $\tau\in (0,\infty]$. Assume that $(X,V,T)$ be a local-in-time solution to the singular system \eqref{TCSUSS} on $t\in (0,\tau)$. Then, $\min_{1\leq i\leq N}T_i(t)$ is monotone increasing and $\max_{1\leq i\leq N}T_i(t)$ is monotone decreasing in $t\in [0,\tau)$. Hence, one has the uniform boundedness of temperature as below.
	\[0<\min_{1\leq i\leq N}T_i^{0}:=T_m^\infty\leq T_i(t) \leq \max_{1\leq i\leq N}T_i^{0}:=T_M^\infty,\quad i=1,\cdots,N,\quad t\in[0,\tau).\] 
\end{proposition}

\begin{remark}\label{R2.1}
	Under the identical initial temperature $T_1^0=\cdots T_N^0=T^0>0$, the model \eqref{TCSUSS} can be reduced to the standard unit-speed Cucker-Smale model with singular communications.
\end{remark} 
We present the previous results for the emergent behaviors of the system \eqref{TCSUS} under regular kernels satisfying \eqref{A-1}.
\begin{proposition}\label{P2.3}
	Let $(X,V,T)$ be a global-in-time solution to \eqref{TCSUS} and assume that there exists a positive constant $D_{X}^\infty$ satisfying \eqref{A-1} and
	\begin{equation*}
	D^2_{V}(0)<\frac{T_m^\infty\phi(D_{X}^\infty)}{2T_M^\infty}\quad\mbox{and}\quad D_{X}(0)+\frac{2T_M^\infty  D_{V}(0)}{\kappa_1{\phi(D_X^\infty)}}< D_{X}^\infty.
	\end{equation*}
	Then, one has $\forall t\in [0,\infty)$,
	\begin{align*}
	D_V^2(t)<2D_V^2(0)\quad \mbox{and}\quad D_X(t)<D_X^\infty.
	\end{align*}As a direct results, one has the following emergence dynamics: 
	\[D_V(t)\leq D_V(0)\exp\left(-\frac{\kappa_1\phi(D_X^\infty)}{2T_M^\infty}\right),\quad D_T(t)\leq D_T(0)\exp\left(-\frac{\kappa_2\zeta(D_X^\infty)}{(T_M^\infty)^2}t\right),\quad \forall t\in [0,\infty).\] 
\end{proposition}

\subsubsection{Collision avoidance via singular communication}

We briefly recap the relationship between the intensity of flocking force and its effects on collision between particles. In order for two particles to collide, they must first approach each other closely. If the flocking force is strong enough to cause a blow-up when particles come close, we can expect flocking to occur before the collision. Additionally, particles that are moving in almost the same direction are unlikely to collide, resulting in collision avoidance. Indeed, the next proposition states that strong singularity (that is, the integral diverges near the origin) deduce the collision avoidance.

\begin{proposition}[\cite{A}]

Let $\{x_i,v_i\}$ be a solution to \eqref{TCS} under $T_i \equiv 1$ and $\phi(r)=r^{-\alpha}$. Suppose that
\begin{center}
	$\alpha \geq 1, \quad$ and $\quad \min_{i \neq j}\|x_i^0-x_j^0\|>0$. 
\end{center}
Then we have the global-in-time collision-less state:
\[
	x_i(t) \neq x_j(t), \quad \forall i,j \in \{1,\cdots,N\}, ~ i \neq j, \quad \forall t \in \bbr_+.
\]
\end{proposition}

The natural question that arises is whether this collision avoidance property holds for model \eqref{TCSUS}. However, the answer is \emph{negative}.

\begin{example}\label{E2.1}
Consider the two-particle model $\{(x_i,v_i,T_i)\}_{i=1,2}$ on the real line governed by \eqref{TCSUS}. We pose
\[
	x_1^0<x_2^0, \quad v_1=1, \quad v_2=-1, \quad T_1=T_2=1.
\]
In this case, the dynamics simplifies into

\begin{equation*}
\begin{cases}
\displaystyle \frac{d{x}_i}{dt} =v_i,\quad t>0,\quad  i=1,\cdots,N,\\
\displaystyle \frac{dv_i}{dt}=\displaystyle\frac{\kappa_1}{N}\sum_{j=1}^{N}\phi(\|x_i-x_j\|)\left({v_j}-v_i^2 v_j \right),\\
\displaystyle (x_i(0),v_i(0))=(x_i^0,v_i^0)\in \mathbb{R}\times\{-1,1\}.
\end{cases}
\end{equation*}
Since each particle always has unit speed, we have $v_i^2 = 1$ for each $i$ and this yields $\dot{v}_i \equiv 0$, which leads to
\[
	v_1 \equiv 1, \quad v_2 \equiv -1.
\]
Therefore, two particles collide at time $(x_2^0-x_1^0)/2$, wether the communication $\phi$ is regular or singular.
\end{example}

As described in Example \ref{E2.1}, although models \eqref{TCS} and \eqref{TCSUS} are phenomenologically similar in that they exhibit flocking and thermal equilibration, their specific operating mechanisms are quite different. It is likely that research on collision avoidance in unit speed models like \eqref{TCSUS} has not been conducted for this reason, and as far as the authors are aware, results of this kind are being addressed for the first time in this paper.

\subsection{Basic estimates} \label{sec:2.2} 
In this subsection, we will show that $\mathcal{A}(v)$ is monotone increasing under an assumption $\mathcal{A}(v)(0)>0$.

\begin{lemma}\label{L2.1}
	Let $\tau\in (0,\infty]$ and $(X,V,T)$ be a local-in-time solution to the singular system \eqref{TCSUSS} on $t\in (0,\tau)$. Then, we can have that
	\[\|v_i(t)\|=1,\quad\forall t\in [0,\infty),\hspace{0.2cm}\forall i=1,\cdots, N.\]
\end{lemma}

\begin{proof}
	We take the inner product to $\eqref{TCSUSS}_2$ with $v_i$ to get the desired estimate.
\end{proof}

Thanks to Proposition \ref{P2.2} and Lemma \ref{L2.1}, the well-posedness of the dynamical system \eqref{TCSUSS} can be guaranteed by the standard Cauchy-Lipschitz theory, provided that collision avoidance between each pair of particles can be ensured on any finite-in-time interval.

\begin{proposition} \label{P2.4}
	Let $\tau\in (0,\infty]$ and $(X,V,T)$ be a local-in-time solution to the singular system \eqref{TCSUSS} with $\mathcal{A}(v(0))>0$ on $t\in (0,\tau)$. Then, one has the velocity-pair angle functional $\mathcal{A}(v)$ is monotone increasing. As a direct consequence, one can obtain
	\[\mathcal{A}(v)(t)\geq \mathcal{A}(v)(0),\quad t\in [0,\tau).\]
\end{proposition}

\begin{proof}
	For any fixed $t\in (0,\tau)$, we define two indices $1\leq i_t, j_t\leq N$ satisfying
	\[\mathcal{A}(v(t)):=\min_{1\leq i,j\leq N}\langle v_i(t),v_j(t)\rangle=\langle v_{i_t},v_{j_t}\rangle.\] 
	Since $\mathcal{A}(v(0))>0$, the following set $\mathcal{S}$ is non-empty:
	\[\mathcal{S}:=\{t\in\mathbb{R}_{>0}~|~\mathcal{A}(v(t))>0\}.\] 
	Now, set $\sup\mathcal{S}:=T^*$. We will show that $T^*=\tau.$ For the proof by contradiction, suppose that \[T^*<\tau,\] which implies $\lim_{t\to T^{*-}}\mathcal{A}(v(t))=0.$  By differentiating $\mathcal{A}(v(t))$ with respect to time $t$, we have
	\begin{align*}
	\frac{d}{dt}\mathcal{A}(v(t))=&\langle \dot{v}_{i_t},v_{j_t}\rangle+\langle \dot{v}_{j_t},v_{i_t}\rangle\\
	=&\frac{\kappa_1}{N}\sum_{k=1}^N\phi(\|x_k-x_{i_t}\|)\frac{\left(\langle v_k,v_{j_t} \rangle-\langle v_k,v_{i_t}\rangle\langle v_{i_t},v_{j_t}\rangle\right)}{T_k}\\
	&+\frac{\kappa_1}{N}\sum_{k=1}^N\phi(\|x_k-x_{j_t}\|)\frac{\left(\langle v_k,v_{i_t} \rangle-\langle v_k,v_{j_t}\rangle\langle v_{i_t},v_{j_t}\rangle\right)}{T_k}, \quad t\in (0,T_*),
	\end{align*} Then the definition of $\mathcal{A}(v(t))$ and strictly positivity of temperature from Proposition \ref{P2.2} yield
	\[\frac{d}{dt}\mathcal{A}(v(t))\geq 0,\quad t\in (0,T_*).\]  Therefore,
	\[\mathcal{A}(v(t))\geq \mathcal{A}(v(0))>0,\quad t\in [0,T_*).\] This contradicts to $\lim_{t\to T^{*-}}\mathcal{A}(v(t))=0$ and thus, which forces \[T^*=\tau.\] This completes the proof by contradiction.
\end{proof}

\begin{remark}\label{R2.2}\emph{\cite{A2}}
In summary, Lemma \ref{L2.1} and Proposition \ref{P2.3} implies
\[0<\mathcal{A}(v(0))\leq \mathcal{A}(v(t))\leq 1,\quad \forall t\in [0,\tau).\]
This fact will be crucially used throughout paper to construct the sufficient framework for the collision avoidance and emergent behaviors of the system \eqref{TCSUSS}.
\end{remark}

\section{Strongly singular interaction kernel $(\alpha\geq 1)$} \label{sec:3}
\setcounter{equation}{0}
In this section, we aim to establish a framework for the global well-posedness of the system \eqref{TCSUSS} by introducing a suitable functional and constructing an admissible set based on the initial data and systemic parameters for the emergence dynamics, as defined in Definition \ref{D1.1}. To accomplish this, we need to derive dissipative differential inequalities with respect to the proper diameters of the position-velocity-temperature. This is essential to observe the dissipative structure of the system $\eqref{TCSUS}$ since the conservation of momentum does not hold in this case. It is worth noting that it is challenging to utilize the configuration vectors with respect to the position-velocity-temperature to induce dissipative differential inequalities due to the unit modulus of each speed in \eqref{TCSUSS}. Before proceeding, we recall the notation of the diameters $D_X$, $D_V$, and $D_T$.
\[D_{X}(t):=\max_{1\leq i,j\leq N}\|x_i(t)-x_j(t)\|,\quad D_{V}(t):=\max_{1\leq i,j\leq N}\|v_i(t)-v_j(t)\|.\]
In particular, since the temperature term is scalar, \[D_T(t):=\max_{1\leq i,j\leq N}|T_i(t)-T_j(t)|.\]

\subsection{Global well-posedness} \label{sec:3.1} 

In this subsection, we will rigorously establish the global well-posedness of \eqref{TCSUSS} under the strongly singular interaction kernel. As argued in Section \ref{sec:2.2}, it is enough to show that no pair of particles collide on any finite-in-time interval. Therefore, we assume that $t_0$ is the first collision time of the system and denote $[l]$ by the set of all particles that collide with the $l$-th particle at time $t_0$:
\[[l]:=\{i\in\{1,\cdots,N\}~|~\|x_l(t)-x_i(t)\|\to 0\quad\mbox{as}\quad t\to t_{0}- \}.\]
Let $\delta$ be a positive real number such that
\[\|x_l(t)-x_i(t)\|\geq \delta>0,\quad\forall~t\in[0,t_0)\quad\mbox{and}\quad\forall i\notin [l].\]
We define the position-velocity diameters and velocity-pair angle with respect to $[l]$ for $t\in [0,t_0)$ for a \emph{subsystem} indexed by $[l]$:
\[D_{X,[l]}(t):=\max_{i,j\in [l]}\|x_i(t)-x_j(t)\|,\quad D_{V,[l]}(t):=\max_{i,j\in [l]}\|v_i(t)-v_j(t)\|,\quad \mathcal{A}_{[l]}(v)=\min_{ i,j\in[l]}\langle v_i,v_j \rangle.\]
For the sake of brevity, we set \[\phi_{ij}:=\phi(\|x_i-x_j\|).\]

\begin{lemma}\label{L3.1}
		Let $(X,V,T)$ be a solution to the dynamical system \eqref{TCSUSS} with 
		\[\alpha\geq1,\quad \mathcal{A}(v)(0)>0\quad \mbox{and} \quad\min_{1\leq i\neq j\leq N}\|x_i(0)-x_j(0)\|>0.\] Then, one can obtain the non-collisional global well-posedness of \eqref{TCSUSS}, that is, 
		\[x_i(t)\neq x_j(t),\quad (i,j)\in \{1,\cdots,N\}^2,\quad i\neq j\quad\mbox{and}\quad\forall t\in [0,\infty).\]
		
		
\end{lemma}


\begin{proof}
	First, we observe
	 \[\frac{d\|x_i-x_j\|^2}{dt}=2\langle x_i-x_j,v_i-v_j\rangle\] and apply the Cauchy–Schwarz inequality to yield,
	\begin{align}\label{C0}
	\left|\frac{D_{X,[l]}(t)}{dt}\right|\leq D_{V,[l]}(t), \quad t\in (0,t_0).
	\end{align}
	 We set two time-dependent indices $M_t, m_t\in [l]$ satisfying
	\[D_{V,[l]}(t):=\|v_{M_t}(t)-v_{m_t}(t)\|,\quad m_t, M_t\in [l].\]
	Using Lemma \ref{L2.1} and equation $\eqref{TCSUSS}_2$, we obtain the following estimate.
	\begin{align*}
	\frac{1}{2}\frac{d{D_{V,[l]}^2}}{dt}&=-\frac{d}{dt}\langle v_{M_t},{v}_{m_t}\rangle=-\langle \dot{v}_{M_t},v_{m_t} \rangle-\langle \dot{v}_{m_t},v_{M_t} \rangle\\
	&=-\left\langle \frac{\kappa_1}{N}\sum_{k=1}^N\phi_{M_tk}\left(\frac{v_k-\langle v_{M_t},v_k\rangle v_{M_t}}{T_k}\right),v_{m_t}\right\rangle\\
	&\quad-\left\langle \frac{\kappa_1}{N}\sum_{k=1}^N\phi_{m_tk}\left(\frac{v_k-\langle v_{m_t},v_k\rangle v_{m_t}}{T_k}\right),v_{M_t}\right\rangle\\
	&=-\left\langle \frac{\kappa_1}{N}\sum_{k\in[l]}\phi_{M_tk}\left(\frac{v_k-\langle v_{M_t},v_k\rangle v_{M_t}}{T_k}\right),v_{m_t}\right\rangle\\
	&\quad-\left\langle \frac{\kappa_1}{N}\sum_{k\in[l]}\phi_{m_tk}\left(\frac{v_k-\langle v_{m_t},v_k\rangle v_{m_t}}{T_k}\right),v_{M_t}\right\rangle\\
	&\quad-\left\langle \frac{\kappa_1}{N}\sum_{k\notin [l]}\phi_{M_tk}\left(\frac{v_k-\langle v_{M_t},v_k\rangle v_{M_t}}{T_k}\right),v_{m_t}\right\rangle\\
	&\quad-\left\langle \frac{\kappa_1}{N}\sum_{k\notin [l]}\phi_{m_tk}\left(\frac{v_k-\langle v_{m_t},v_k\rangle v_{m_t}}{T_k}\right),v_{M_t}\right\rangle:=\mathcal{I}_1+\mathcal{I}_2+\mathcal{I}_3+\mathcal{I}_4.
	\end{align*}
	$\bullet$(The estimate of $\mathcal{I}_1+\mathcal{I}_2$) It follows from Lemma \ref{L2.1} and Remark \ref{R2.2} that
	\[\langle v_k,v_{m_t} \rangle \geq \langle v_k,v_{M_t} \rangle\langle v_{m_t},v_{M_t} \rangle\quad\mbox{and} \quad\langle v_k,v_{M_t} \rangle \geq \langle v_k,v_{m_t} \rangle\langle v_{m_t},v_{M_t} \rangle.\]
	We utilize above inequalities and Proposition \ref{P2.2} to have

	\begin{align*}
	\mathcal{I}_1+\mathcal{I}_2&\leq -\left\langle \frac{\kappa_1}{N}\sum_{k\in[l]}\phi_{M_tk}\left(\frac{v_k-\langle v_{M_t},v_k\rangle v_{M_t}}{T_k}\right),v_{m_t}\right\rangle\\
	&\quad-\left\langle \frac{\kappa_1}{N}\sum_{k\in[l]}\phi_{m_tk}\left(\frac{v_k-\langle v_{m_t},v_k\rangle v_{m_t}}{T_k}\right),v_{M_t}\right\rangle\\
	&\leq -\frac{\kappa_1\phi(D_{X,[l]})}{NT_M^\infty}\sum_{k\in[l]}\left(\langle v_k,v_{m_t} \rangle-\langle v_k,v_{M_t} \rangle\langle v_{m_t},v_{M_t}\rangle\right)\\
	&\quad-\frac{\kappa_1\phi(D_{X,[l]})}{NT_M^\infty}\sum_{k\in[l]}\left(\langle v_k,v_{M_t} \rangle-\langle v_k,v_{m_t} \rangle\langle v_{m_t},v_{M_t}\rangle\right)\\
	&=-\frac{\kappa_1\phi(D_{X,[l]})}{NT_M^\infty}\sum_{k\in[l]}\left(\langle v_k,v_{M_t} \rangle-\langle v_k,v_{M_t} \rangle\langle v_{m_t},v_{M_t}\rangle\right)\\
	&\quad-\frac{\kappa_1\phi(D_{X,[l]})}{NT_M^\infty}\sum_{k\in[l]}\left(\langle v_k,v_{m_t} \rangle-\langle v_k,v_{m_t} \rangle\langle v_{m_t},v_{M_t}\rangle\right)\\
	&=-\frac{\kappa_1\phi(D_{X,[l]})D_{V,[l]}^2}{2NT_M^\infty}\sum_{k\in[l]}\left(\langle v_k,v_{M_t} \rangle+\langle v_k,v_{m_t}\rangle\right)\leq -\frac{\kappa_1|[l]|\mathcal{A}_{[l]}(v)(0)}{NT_M^\infty}\phi(D_{X,[l]})D_{V,[l]}^2,
	\end{align*}where $|[l]|$ is the cardinal number of $[l]$. We used Remark \ref{R2.2} for the last inequality.\\
	\newline
	$\bullet$(The estimate of $\mathcal{I}_3+\mathcal{I}_4$)
	We use the triangle inequality, Cauchy's inequality, definition of $\delta$ and Proposition \ref{P2.2} to obtain
	\begin{align*}
	\mathcal{I}_3+\mathcal{I}_4&\leq \frac{\kappa_1\phi(\delta)}{NT_m^\infty}\sum_{k\notin[l]}\left|\left\langle v_k,v_{m_t} \right\rangle-\langle v_k,v_{M_t} \rangle\langle v_{m_t},v_{M_t}\rangle\right|\\
	&\quad+\frac{\kappa_1\phi(\delta)}{NT_m^\infty}\sum_{k\notin[l]}\left|\langle v_k,v_{M_t} \rangle-\langle v_k,v_{m_t} \rangle\langle v_{m_t},v_{M_t}\rangle\right|\\
	&\leq \frac{\kappa_1\phi(\delta)}{NT_m^\infty}\sum_{k\notin[l]} \left( \left|\langle v_k,v_{M_t}-v_{m_t} \rangle\right| + (1-\langle v_{M_t},v_{m_t}\rangle )|\langle v_k,v_{M_t} \rangle| \right)\\
	&\quad+\frac{\kappa_1\phi(\delta)}{NT_m^\infty}\sum_{k\notin[l]} \left( \left|\langle v_k,v_{M_t}-v_{m_t} \rangle\right| + (1-\langle v_{M_t},v_{m_t}\rangle )|\langle v_k,v_{M_t} \rangle| \right)\\
	&\leq \frac{2\kappa_1\phi(\delta)}{NT_m^\infty}\sum_{k\notin[l]} \left( D_{V,[l]}+\frac{D_{V,[l]}^2}{2} \right)\leq \frac{4\kappa_1\phi(\delta)}{NT_m^\infty}\sum_{k\notin[l]}  D_{V,[l]}=\frac{4\kappa_1(N-|[l]|)\phi(\delta)}{NT_m^\infty}D_{V,[l]}.
	\end{align*}
	Therefore, combine the estimates on $\mathcal{I}_1+\mathcal{I}_2$ and $\mathcal{I}_3+\mathcal{I}_4$ leads to
	\begin{align}\label{C1}
	\begin{aligned}
    \frac{dD_{V,[l]}}{dt} &\leq -\frac{\kappa_1|[l]|\mathcal{A}_{[l]}(v)(0)}{NT_M^\infty}\phi(D_{X,[l]})D_{V,[l]}+\frac{4\kappa_1(N-|[l]|)\phi(\delta)}{NT_m^\infty}\\
    &=: -C_1\phi(D_{X,[l]})D_{V,[l]}+C_2,\quad \mbox{a.e.}~t\in (0,t_0).
    \end{aligned}
	\end{align}Integrating to both sides of \eqref{C1} from $s$ to $t$ for $0\leq s\leq t <t_0$, one attain that
	
	\begin{align}\label{key}
	\int_{s}^{t}\phi(D_{X,[l]})D_{V,[l]} ds \leq \frac{D_{V,[l]}(s)+C_2(t-s)}{C_1}.
	\end{align}
	
%
	\noindent On the other hand, let $\Phi$ be the primitive of the strongly singular weight $\phi$ with $\alpha\geq 1$. Then, for fixed $t_1>0$,
	\begin{align*}
	\Phi(t):=\Phi(t;t_1):=\int_{t_1}^{t} \phi(u)du=
	{\begin{cases}
	\displaystyle\log\frac{t}{t_1},\quad\quad\quad\quad\quad\quad\quad\hspace{0.3cm} \mbox{if}\quad \alpha=1,\\
	\newline
	\displaystyle\frac{1}{1-\alpha}\left(t^{1-\alpha}-{t_1}^{1-\alpha}\right),\quad \mbox{if}\quad \alpha>1.
	\end{cases}}
	\end{align*}Therefore, it follows from \eqref{C0} that for $0\leq s\leq t<t_0$,
	
	\begin{align*}
	\left|\Phi(D_{X,[l]}(t))\right| &\leq \left|\int_{s}^{t}\frac{d}{du} \Phi(D_{X,[l]}(u)) du\right|+\left|\Phi(D_{X,[l]}(s))\right|\\
	&= \left|\int_{s}^{t}\phi(D_{X,[l]}(u))\left(\frac{d}{du}D_{X,[l]}(u)\right) du\right|+\left|\Phi(D_{X,[l]}(s))\right|\\
	&\leq \int_{s}^{t}\phi(D_{X,[l]}(u))\left|\frac{d}{du}D_{X,[l]}(u)\right| du+\left|\Phi(D_{X,[l]}(s))\right|\\
	&\leq  \int_{s}^{t}\phi(D_{X,[l]}(u))D_{V,[l]}(u) du+\left|\Phi(D_{X,[l]}(s))\right|,
	\end{align*}
	combining this with \eqref{key} yields
%
	
	
	\begin{align*}
	&\left|\Phi(D_{X,[l]}(t))\right|\leq  \left|\Phi(D_{X,[l]}(s))\right|+ \frac{D_{V,[l]}(s)+C_2(t-s)}{C_1},\quad 0\leq s\leq t<t_0.
	\end{align*}
%
	\noindent We take $s=0$ and $t\to t_0-$ to the above inequality to acquire
	\[\lim_{t\to t_0-}\left|\Phi(D_{X,[l]}(t))\right|\leq \left|\Phi(D_{X,[l]}(0))\right|+\frac{C_2}{C_1}\left(t_0+D_{V,[l]}(0)\right)<\infty,\]
	which gives a contradiction to the definition of $t_0$ and $\alpha\geq 1$, i.e.,
	\[\lim_{t\to t_0-}\left|\Phi(D_{X,[l]}(t))\right|=\infty.\]
	This proves the collision-avoidance property of the system \eqref{TCSUSS}, i.e. 
	\[x_i(t)\neq x_j(t),\quad (i,j)\in \{1,\cdots,N\}^2\quad \mbox{and}\quad\forall t\in [0,\infty).\]
	In particular, this guarantees the global well-posedness of \eqref{TCSUSS} by the standard Cauchy-Lipschitz theory.
\end{proof}

\begin{remark}
	Under the weak singularity $0<\alpha<1$, the limit $\lim_{t\to t_0-}\left|\Phi(D_{X,[l]}(t))\right|=\infty$ does not hold. In terms of $\phi$, its the integrability near the origin of is crucial for collision avoidance. Indeed, the integrability at the origin is a \emph{characteristic} property for collision avoidance under $\mathcal{A}(0)>0$. We will study this property further in Section \ref{sec:4.1}.
	\end{remark}

\subsection{Emergent dynamics by bootstrapping argument} \label{sec:3.2} 
This subsection provides a study of the sufficient framework for the emergent dynamics of the system \eqref{TCSUSS} under $\alpha\geq1$, leveraging the insights obtained from Lemma \ref{L3.1}. For this, we employ the functional $\Psi$ defined by
\[\Psi_{ij}(t):=\frac{\zeta(\|x_i-x_j\|)}{N},~ i\neq j,\quad \Psi_{ii}(t):=\zeta\left(\min_{1\leq i\neq j\leq N}\|x_i-x_j\|\right)-\frac{\sum_{j=1,j\neq i}^{N}\zeta(\|x_i-x_j\|)}{N}.
\] Then, we can observe that $\Psi_{ij}$ satisfies the following properties: 
\[\Psi_{ij}\geq \frac{\zeta_{ij}}{N},\quad \sum_{j=1}^{N}\Psi_{ij}=\zeta\left(\min_{1\leq i,j\leq N}\|x_i-x_j\|\right),\quad \sum_{j=1}^{N}\Psi_{ij}\left(\frac{1}{T_i}-\frac{1}{T_j}\right)=\sum_{j=1}^{N}\frac{\zeta_{ij}}{N}\left(\frac{1}{T_i}-\frac{1}{T_j}\right),\]
where the simplified notation $\zeta_{ij}$ refers to
\[\zeta_{ij}:=\zeta(\|x_i-x_j\|).\]
\begin{theorem}\label{T3.1}
	Let $(X,V,T)$ be a solution to the dynamical system \eqref{TCSUSS} with 
	\[\alpha\geq1,\quad \mathcal{A}(v)(0)>0\quad \mbox{and} \quad\min_{1\leq i\neq j\leq N}\|x_i(0)-x_j(0)\|>0.\]
	Further assume that there exists a positive constant $D_{X}^\infty$ satisfying
	\begin{equation}\label{priori2}
	D_{X}(0)+\frac{T_M^\infty}{\kappa_1\mathcal{A}(v)(0){\phi(D_X^\infty)}}\cdot D_{V}(0)< D_{X}^\infty.
	\end{equation}
	Then, one has the following asymptotic flocking and temperature equilibrium: 
	\begin{enumerate}\label{C5}
	\item(Group formation) $\displaystyle D_X(t)<D_X^\infty$,
	\vspace{0.3cm}
	\item(Velocity alignment) $\displaystyle D_V(t)\leq D_V(0)\exp\left(-\frac{\kappa_1\mathcal{A}(v)(0)\phi(D_X^\infty)}{T_M^\infty}t\right)$,
	\item(Temperature equilibrium) $\displaystyle D_T(t)\leq D_T(0)\exp\left(-\frac{\kappa_2\zeta(D_X^\infty)}{(T_M^\infty)^2}t\right),\quad\forall t\in [0,\infty).$
	\end{enumerate}
\end{theorem}

\begin{proof}
	From \eqref{priori2}, the following set   
	\[S:=\{t>0~|~D_X(s)<D_X^\infty,~ \forall s\in(0,t)\}\]
	is non-empty and $t^*:=\sup S>0$ is well defined. If $t^* < \infty$, then $D_X(t^*)=D_X^\infty.$ Now we claim 
	\[ t^*=+\infty. \]
	For the proof by contradiction, suppose that $t^*<+\infty$. Then by replacing $[l]$ with $\{1,\cdots,N\}$ in the proof of Lemma \ref{L3.1}, it follows from Lemma \ref{L3.1} that 
	\[\left|\frac{dD_X}{dt}\right|\leq D_V,\quad \frac{dD_V}{dt}\leq-\frac{\kappa_1\mathcal{A}(v)(0)}{T_M^\infty}\phi(D_{X})D_{V}\leq -\frac{\kappa_1\mathcal{A}(v)(0)\phi(D_{X}^\infty)}{T_M^\infty}D_{V}\quad \text{a.e.}~t\in (0,t^*).\] 
	Then the Gro\"nwall lemma implies
	\[D_V(t)\leq D_V(0)\exp\left(-\frac{\kappa_1\mathcal{A}(v)(0)\phi(D_X^\infty)}{T_M^\infty}t\right),\quad \forall t\in [0,t^*].\]
	Since
	\begin{align*}
	D_{X}(t)&= D_{X}(0)+\int_{0}^{t}\frac{dD_{X}(s)}{ds}ds \leq D_{X}(0)+\int_{0}^{t}D_{V}(s)ds\\
	&\leq D_{X}(0)+\int_{0}^{t}D_{V}(0)\exp\left(-\frac{\kappa_1\mathcal{A}(v)(0)\phi(D_X^\infty)}{T_M^\infty}s\right)ds\\
	&\leq D_{X}(0)+\frac{T_M^\infty}{\kappa_1\mathcal{A}(v)(0){\phi(D_X^\infty)}}\cdot D_{V}(0)< D_{X}^\infty, \quad \forall t\in [0,t^*],
	\end{align*}one has $D_X(t^*)<D_X^\infty$. Therefore, it gives a contradiction, that is, $t^*=\infty$. Accordingly, we have
	\[D_X(t)<D_X^\infty,\quad D_V(t)\leq D_V(0)\exp\left(-\frac{\kappa_1\mathcal{A}(v)(0)\phi(D_X^\infty)}{T_M^\infty}t\right),\quad t\in [0,\infty).\]
	From now on, we move on to the temperature equilibrium estimates. First of all, select two indices $M_t$ and $m_t$ depending on time $t$ such that
	\[D_T(t)=T_{M_t}(t)-T_{m_t}(t),\quad 1\leq m_t, M_t\leq N.\] Then, it follows from the properties of $\Psi_{ij}$ and $\eqref{TCSUSS}_3$ that
	\begin{align*}
	&\frac{dD_T}{dt}=\dot{T}_{M_t}-\dot{T}_{m_t}\\
	&=\frac{\kappa_2}{N}\sum_{k=1}^N\zeta_{M_tk}\left(\frac{1}{T_{M_t}}-\frac{1}{T_k}\right)-\frac{\kappa_2}{N}\sum_{k=1}^N\zeta_{m_tk}\left(\frac{1}{T_{m_t}}-\frac{1}{T_k}\right)\\
	&={\kappa_2}\sum_{k=1}^N\Psi_{M_tk}\left(\frac{1}{T_{M_t}}-\frac{1}{T_k}\right)-{\kappa_2}\sum_{k=1}^N\Psi_{m_tk}\left(\frac{1}{T_{m_t}}-\frac{1}{T_k}\right)\\
	&=\kappa_2\zeta\left(\min_{1\leq i,j\leq N}\|x_i-x_j\|\right)\left(\frac{1}{T_{M_t}}-\frac{1}{T_{m_t}}\right)-\kappa_2\sum_{k=1}^N\frac{1}{T_k}\left(\Psi_{M_tk}-\Psi_{m_tk}\right)\\
	&=\kappa_2\zeta\left(\min_{1\leq i,j\leq N}\|x_i-x_j\|\right)\left(\frac{1}{T_{M_t}}-\frac{1}{T_{m_t}}\right)\\
	&\quad-\kappa_2\sum_{k=1}^N\frac{1}{T_k}\left(\Psi_{M_tk}-\min(\Psi_{M_tk},\Psi_{m_tk})+\min(\Psi_{M_tk},\Psi_{m_tk})-\Psi_{m_tk}\right)\\
	&\leq \kappa_2\zeta\left(\min_{1\leq i,j\leq N}\|x_i-x_j\|\right)\left(\frac{1}{T_{M_t}}-\frac{1}{T_{m_t}}\right)+\frac{\kappa_2}{T_{m_t}}\sum_{k=1}^N\left(\Psi_{m_tk}-\min(\Psi_{M_tk},\Psi_{m_tk})\right)\\
	&\quad-\frac{\kappa_2}{T_{M_t}}\sum_{k=1}^N\left(\Psi_{M_tk}-\min(\Psi_{M_tk},\Psi_{m_tk})\right)\\
	&=-\kappa_2\left(\frac{1}{T_{m_t}}-\frac{1}{T_{M_t}}\right)\sum_{k=1}^N\left(\min(\Psi_{M_tk},\Psi_{m_tk})\right)\leq -\frac{\kappa_2D_T}{(T_M^\infty)^2}\sum_{k=1}^N\left(\min(\Psi_{M_tk},\Psi_{m_tk})\right)\\
	&\leq -\frac{\kappa_2\zeta(D_X)}{(T_M^\infty)^2}D_T\leq -\frac{\kappa_2\zeta(D_X^\infty)}{(T_M^\infty)^2}D_T,\quad \mbox{a.e.}~ t\in (0,\infty).
	\end{align*}The Gro\"nwall lemma induce the following exponential thermal equilibrium:
	
	\[D_T(t)\leq D_T(0)\exp\left(-\frac{\kappa_2\zeta(D_X^\infty)}{(T_M^\infty)^2}t\right),\quad\forall t\in [0,\infty).\]
\end{proof}

\subsection{Emergent dynamics by Lyapunov functional approach} \label{sec:3.3}
In this subsection, we introduce an alternative approach for obtaining the emergence dynamics of the model \eqref{TCSUSS} using a suitable Lyapunov functional. The detail is provided in the following theorem.

\begin{theorem}\label{T3.2}
Let $(X,V,T)$ be a solution to the dynamical system \eqref{TCSUSS} with 
\[\alpha\geq1,\quad \mathcal{A}(v)(0)>0\quad \mbox{and} \quad\min_{1\leq i\neq j\leq N}\|x_i(0)-x_j(0)\|>0.\]
Further assume that 
\begin{equation}\label{priori}
D_V(0)< \frac{\kappa_1\mathcal{A}(v)(0)}{T_M^\infty}\int_{D_X(0)}^\infty \phi(s)ds.
\end{equation}
Then, one has the following asymptotic flocking and temperature equilibrium: there exists a strictly positive number $D_X^\infty>0$ satisfying the following assertions.
\begin{enumerate}\label{C6}
	\item(Group formation) $\displaystyle D_X(t)\leq D_X^\infty$,
	\vspace{0.3cm}
	\item(Velocity alignment) $\displaystyle D_V(t)\leq D_V(0)\exp\left(-\frac{\kappa_1\mathcal{A}(v)(0)\phi(D_X^\infty)}{T_M^\infty}t\right)$,
	\item(Temperature equilibrium) $\displaystyle D_T(t)\leq D_T(0)\exp\left(-\frac{\kappa_2\zeta(D_X^\infty)}{(T_M^\infty)^2}t\right),\quad\forall t\in [0,\infty).$
\end{enumerate}
\end{theorem}

\begin{proof}
Firstly, use the proofs of Lemma $\ref{L3.1}$ and Theorem \ref{T3.1} to get that
 
\begin{align}\label{C-0}
\left|\frac{dD_X}{dt}\right|\leq D_V,\quad \frac{dD_V}{dt}\leq-\frac{\kappa_1\mathcal{A}(v)(0)}{T_M^\infty}\phi(D_{X})D_{V}.
\end{align}
Now, employ the following Lyapunov functional as below:
\begin{align*} 
\mathcal{L}_{\pm}(D_X,D_V):= D_V{\pm}\frac{\kappa_1\mathcal{A}(v)(0)}{T_M^\infty}\Phi(D_X),
\end{align*}where $\Phi(t):=\int_{D_X(0)}^{t}\phi(s)ds$ for a fixed $t_1>0$. Then, we can derive by using \eqref{C-0}

\begin{align}
\begin{aligned}\label{C-1}
\frac{d}{dt}\mathcal{L}_{\pm}(D_X,D_V)&=\frac{dD_V}{dt}\pm \frac{\kappa_1\mathcal{A}(v)(0)}{T_M^\infty}\frac{dD_X}{dt}\phi(D_X)\\
&\leq -\frac{\kappa_1\mathcal{A}(v)(0)}{T_M^\infty}\phi(D_X)D_V\pm \frac{\kappa_1\mathcal{A}(v)(0)}{T_M^\infty}\frac{dD_X}{dt}\phi(D_X)\\
&= \frac{\kappa_1\mathcal{A}(v)(0)}{T_M^\infty}\phi(D_X)\left(-D_V\pm\frac{dD_X}{dt}\right)\leq 0.
\end{aligned}
\end{align} In here, use \eqref{C-1}, that is $\mathcal{L}_{\pm}(D_X(t),D_V(t))\leq \mathcal{L}_{\pm}(D_X(0),D_V(0))$ to obtain
\begin{align*}
D_V(t)+\frac{\kappa_1\mathcal{A}(v)(0)}{T_M^\infty}\left|\int_{D_X(0)}^{D_X(t)}\phi(s)ds\right|\leq D_V(0), 
\end{align*}which implies 
\begin{align}\label{C-2}
\frac{\kappa_1\mathcal{A}(v)(0)}{T_M^\infty}\left|\int_{D_X(0)}^{D_X(t)}\phi(s)ds\right|\leq D_V(0).
\end{align}
On the other hand, by {\it{a priori}} assumption \eqref{priori} one has

\[D_V(0)< \frac{\kappa_1\mathcal{A}(v)(0)}{T_M^\infty}\int_{D_X(0)}^\infty \phi(s)ds.\] Hence, there is the smallest positive real number $D_X^\infty$ such that
\[D_V(0)=\frac{\kappa_1\mathcal{A}(v)(0)}{T_M^\infty}\int_{D_X(0)}^{D_X^\infty} \phi(s)ds.\] Therefore, it follows from \eqref{C-2} that
\[D_X(t)\leq D_X^\infty,\quad\forall t\in [0,\infty).\] 
Thanks to the above estimate and \eqref{C-0}, it leads to
\[\frac{dD_V(t)}{dt}\leq -\frac{\kappa_1\mathcal{A}(v)(0)\phi(D_X^\infty)}{T_M^\infty}D_V(t)\] and then, apply Gro\"nwall's lemma to conclude that
\[D_V(t)\leq D_V(0)\exp\left(-\frac{\kappa_1\mathcal{A}(v)(0)\phi(D_X^\infty)}{T_M^\infty}t\right),\quad t\in [0,\infty).\] Finally, for the thermal equilibrium estimate one can prove it in the same way as Theorem $\ref{T3.1}_3$. Thus,

\[D_T(t)\leq D_T(0)\exp\left(-\frac{\kappa_2\zeta(D_X^\infty)}{(T_M^\infty)^2}t\right),\quad\forall t\in [0,\infty).\]
We complete this proof.

\end{proof}

\begin{remark}\label{R3.3}
If $\alpha\geq 1$, then it follows that \[\int_{D_X(0)}^\infty \phi(s)ds<\infty.\] Therefore, the {\it{a priori}} condition \eqref{priori} cannot be removed in this case. However, when $0<\alpha<1$ and there are no collisions, the condition \eqref{priori} can be removed since we have \[\int_{D_X(0)}^\infty \phi(s)ds=\infty.\]
\end{remark}

\section{Weakly singular interaction kernel $(0<\alpha<1)$} \label{sec:4}
\setcounter{equation}{0}
In this section, we establish a sufficient framework for the case of weakly singular interaction kernel with $0<\alpha<1$ in \eqref{TCSUSS}. Unlike the case of strongly singular weight studied in Section \ref{sec:3}, more restrictive conditions are required to ensure the global well-posedness of \eqref{TCSUSS} under $0<\alpha<1$, even though we assume the conditions in Lemma \ref{L3.1}:
\[\mathcal{A}(v)(0)>0,\quad \min_{ 1\leq i,j\leq N}|x_i(0)-x_j(0)|>0.\]
For further details, refer to the forthcoming Section \ref{sec:4.1}.

\subsection{Existence of collisions in finite time} \label{sec:4.1}
In this subsection, we establish sufficient conditions for the occurrence of collisions in finite time for the weakly singular system \eqref{TCSUSS} with $0<\alpha<1$ and $N=2$ in two dimensions. For simplicity, assume that
\[T_1(0)=T_2(0):=T^0>0.\]
Then, we use Remark \ref{R2.1} to express the $2$-particle equation \eqref{TCSUSS} explicitly:

\begin{equation}
\begin{cases} \label{TCSUSS2}
\vspace{0.2cm}
\displaystyle \frac{d{x}_1}{dt} =v_1,\quad\frac{d{x}_2}{dt} =v_2,\quad t>0,\\
\vspace{0.4cm}
\displaystyle \frac{dv_1}{dt}=\displaystyle\frac{\kappa_1\left({v_2-\langle v_2,v_1\rangle v_1}\right)}{2\|x_1-x_2\|^\alpha},\quad \displaystyle \frac{dv_2}{dt}=\displaystyle\frac{\kappa_1\left({v_1-\langle v_2,v_1\rangle v_2}\right)}{2\|x_1-x_2\|^\alpha},\\
\displaystyle (x_1(0),v_1(0),x_2(0),v_2(0))=(x_1^0,v_1^0,x^0_2,v_2^0)\in \mathbb{R}^{2}\times\mathbb{S}^{1}\times\mathbb{R}^{2}\times\mathbb{S}^{1}.
\end{cases}
\end{equation}

\begin{remark}
	When we consider the system \eqref{TCSUSS2} on $\mathbb{R}^1$, i.e., $d=1$, it follows that $v_i\equiv1~\mbox{or}~v_i\equiv-1$ for $i=1,\cdots,N$. Then, each pair of inter-particles along the $N$-particle flow \eqref{TCSUSS} on $\mathbb{R}^1$ with non-collisional initial-position data and $\mathcal{A}(v)(0)>0$ do not always collide with each other regardless of the degree of singularity $\alpha$. Indeed, 
	\[v_i\equiv1,~\forall i\in\{1,\cdots,N\}\quad\mbox{or}\quad v_i\equiv-1,~\forall i\in\{1,\cdots,N\}.\]
	Therefore, in the case of $\mathcal{A}(v)(0)>0$, a collision in the $2$-particle system \eqref{TCSUSS2} with weak singular kernel can only occur for $d\geq 2$.
\end{remark}

\begin{proposition}\label{P4.1}
Let $(X,V,T)$ be a solution to the $2$-particle system \eqref{TCSUSS2} such that 
\begin{align}\label{d0}
0<\alpha<1, \quad d=2,\quad x_1(0)\neq x_2(0),\quad \mathcal{A}(v)(0)>0.
\end{align} Then, there exist sufficient conditions only in terms of initial data and system parameters satisfying a finite-in-time collision. That is, there is a strictly positive time $t_0\in (0,\infty)$ such that
\[x_1(t_0)=x_2(t_0).\]
\end{proposition}

\begin{proof}
	For the proof by contradiction, suppose that a unique global solution $(X,V,T)$ is well defined, which is equivalent to the non-collision on any finite time interval of \eqref{TCSUSS2'}. Throughout the proof, under the canonical identification $\bbc \cong \bbr^2$, we identify
	\[
		e^{i\theta} \cong (\cos\theta, \sin\theta),
	\]
	and for $n=1,2$, $n$-th component of $x_k$ (resp. $v_k$) will be denoted as $x_k^n$ (resp. $v_k^n$). Note that $x_k^0$ refers to the initial state $x_k(0)$.
	\\
	\newline
	Let $v_1=e^{i\theta_1}\in \mathbb{S}^1$ and $v_2=e^{i\theta_2}\in\mathbb{S}^1$, where 
	\[\theta_i:[0,\infty)\rightarrow \mathbb{R},\quad \theta_i\in C_1([0,\infty)),\quad i=1,2.\]
	Then, it follows that \eqref{TCSUSS2} can be reformulated to \eqref{TCSUSS2'} as follows.
	\begin{equation}
	\begin{cases} \label{TCSUSS2'}
	\vspace{0.2cm}
	\displaystyle \frac{d{x}_1}{dt} =e^{i\theta_1},\quad\frac{d{x}_2}{dt} =e^{i\theta_2},\quad t>0,\\
	\vspace{0.3cm}
	\displaystyle \frac{d\theta_1}{dt}=\displaystyle\frac{\kappa_1\sin(\theta_2-\theta_1)}{2\|x_1-x_2\|^\alpha},\quad \displaystyle \frac{d\theta_2}{dt}=\displaystyle\frac{\kappa_1\sin(\theta_1-\theta_2)}{2\|x_1-x_2\|^\alpha},\\
	\displaystyle (x_k(0),v_k(0),)=(x_k^0,e^{i\theta_k^0}), \quad
	\quad v_k(t)=(\cos\theta_k, \sin\theta_k), \quad k=1,2.
	\end{cases}
	\end{equation}
    Since $\dot{\theta}_1(t)+\dot{\theta}_2(t)=0,~t\geq 0$, one can assume that without loss of generality,
	\begin{align}\label{D-1}
	{\theta}_1(0)+{\theta}_2(0)=0\quad \mbox{and thus,} \quad{\theta}_1(t)+{\theta}_2(t)=0,\quad t\in [0,\infty).
	\end{align}
	Hence, one has that
	
	\begin{align*}
	\begin{aligned}
	x_1(t)-x_2(t)&= x_1(0)-x_2(0)+\int_{0}^t(v_1(s)-v_2(s)) ds\\
	&= x_1(0)-x_2(0)+\int_{0}^t(e^{i\theta_1(s)}-e^{i\theta_2(s)}) ds\\
	&= x_1(0)-x_2(0)-2\int_{0}^t(0,\sin(\theta_2)) ds,\\
	v_1(t)-v_2(t)&= -2(0, \sin(\theta_2(t))).
	\end{aligned}
	\end{align*}
	Now, we set
	\begin{align}\label{D-0}
	x_1^1(0)=x_2^1(0),\quad x_2^2(0)>x_1^2(0)\quad \mbox{and } \quad v_1^2(0)>v_2^2(0).
	\end{align}
	Then, by the well-posedness of \eqref{TCSUSS2'}, we can obtain that
	\begin{align}\label{D-2}
	x_2^2(t)>x_1^2(t),\quad t\in [0,\infty),
	\end{align}
    since \eqref{D-0} implies that
	\begin{align}\label{D-3}
	x_2(t)-x_1(t)=\left(0,~x^2_2(0)-x^2_1(0)+2\int_0^t \sin (\theta_2) ds\right).
	\end{align}
	Then we use $\eqref{TCSUSS2}_2$, \eqref{D-2} and \eqref{D-3} to have that
	\begin{align}\label{D-4}
	\frac{d}{dt}(v_2^2-v_1^2)=
	\underbrace{-\frac{\kappa_1(1+\langle v_1,v_2 \rangle)}{2(x_2^2-x_1^2)^{\alpha}}}_{<0}(v_2^2-v_1^2).
	\end{align}
	From $\mathcal{A}(v)(0)>0$ and \eqref{D-2}, a coefficient of  $v_2^2-v_1^2$ is bounded above by negative constant in any finite time. Since we are assuming the non-collisional state, the uniqueness of a solution is guaranteed and $v_2^2-v_1^2$ is nonzero in any finite time:
	\[v_1^2(t)\neq v_2^2(t)\quad \mbox{that is,} \quad v_1^2(t)> v_2^2(t),\quad \forall t\in (0,\infty).\]
	Then, it follows from \eqref{D-4} that
	\begin{align*}
	\frac{d}{dt}(v_2^2-v_1^2) &= -\frac{\kappa_1(1+\langle v_1,v_2 \rangle)}{2(x_2^2-x_1^2)^{\alpha}}(v_2^2-v_1^2)
	\leq -\frac{\kappa_1(1+\mathcal{A}(v)(0))}{2(x_2^2-x_1^2)^{\alpha}}(v_2^2-v_1^2)\\
	&=-\frac{\kappa_1(1+\mathcal{A}(v)(0))}{2(x_2^2-x_1^2)^{\alpha}}\frac{d(x_2^2-x_1^2)}{dt}=-\frac{\kappa_1(1+\mathcal{A}(v)(0))}{2(1-\alpha)}\frac{d(x_2^2-x_1^2)^{1-\alpha}}{dt}.
	\end{align*}
	Therefore, if we pose
	\begin{align}\label{D-5}
		v_1^2(0)-v_2^2(0)=\frac{\kappa_1(1+\mathcal{A}(v)(0))}{2(1-\alpha)}(x_2^2(0)-x_1^2(0))^{1-\alpha},
	\end{align}
	then a direct integration yields
	\[
		\frac{d(x_2^2-x_1^2)}{dt}\leq v_2^2-v_1^2\leq -\frac{\kappa_1(1+\mathcal{A}(v)(0))}{2(1-\alpha)}(x_2^2-x_1^2)^{1-\alpha} =: -a(x_2^2-x_1^2)^{1-\alpha}.
	\] 
	We then use the Comparison principle to see
	\[
		x_2^2(t)-x_1^2(t)
		\leq (a\alpha)^{\frac{1}{\alpha}}\left(\frac{(x_2^2(0)-x_1^2(0))^\alpha}{a\alpha}-t\right)^{\frac{1}{\alpha}}.
	\]
	Therefore, $x_2^2-x_1^2$ becomes zero in some finite time, which completes the proof by contradiction.
\end{proof}

\begin{remark}
	Note that the sufficient framework consisting of a priori assumptions \eqref{d0}, \eqref{D-1}, \eqref{D-0} and \eqref{D-5} says that it is necessary to give more restrictive conditions for the collision avoidance of \eqref{TCSUSS} with the weakly singular kernel. To solve this problematic issue, refer to Section \ref{sec:4.2}.
\end{remark}


\subsection{Global well-posedness and emergent dynamics} \label{sec:4.2} 
In this subsection, we present the sufficient framework for the well-posedness of solution under weakly singular communication. In fact, we prove even stronger result; we will demonstrate strict spacing between the particles for any degree of singularity. We will use the emergence dynamics and proofs described in Theorem \ref{T3.1} and Theorem \ref{T3.2} to guarantee its non-collisional phenomenon and emergent behaviors. 

\begin{theorem}\label{T4.1}
	Let $(X,V,T)$ be a solution to the dynamical system \eqref{TCSUSS} with 
	\[\alpha>0,\quad \mathcal{A}(v)(0)>0\quad \mbox{and} \quad\min_{1\leq i\neq j\leq N}\|x_i(0)-x_j(0)\|>0.\]
	Suppose that either
	\begin{enumerate}
	\item there exist a positive constant $D_X^\infty$ and a natural number $k\in\{1,\cdots,d\}$ satisfying
	\begin{align}\begin{aligned}\label{D-6}
    &\frac{T_M^\infty D_{V}(0)}{\kappa_1\mathcal{A}(v)(0){\phi(D_X^\infty)}}\\
    &\hspace{.4cm} <\min\left(
    \max\left\{ D_X^\infty-D_X(0),\frac{1}{\phi(D_X^\infty)}\int_{D_X(0)}^{D_X^\infty} \phi(s)ds \right\},
    \min_{1\leq i\neq j \leq N}|x^k_i(0)-x^k_j(0)|
    \right),
	\end{aligned}\end{align}
	\item or $\alpha > 1$ and there exist a positive constant $D_X^\infty$ satisfying
		\begin{align}\label{D-7}
    \frac{T_M^\infty D_{V}(0)}{\kappa_1\mathcal{A}(v)(0){\phi(D_X^\infty)}}
    <\max\left\{ D_X^\infty-D_X(0),\frac{1}{\phi(D_X^\infty)}\int_{D_X(0)}^{D_X^\infty} \phi(s)ds \right\}.
	\end{align}

	\end{enumerate}
	Then a solution achieve collision avoidance, asymptotic flocking, and temperature equilibrium. Furthermore, a strict spacing between the particles is guaranteed;
	\begin{align*}
	\inf_{t \geq 0}\min_{ 1\leq i\neq j\leq N}\|x_i(t)-x_j(t)\| > 0.
	\end{align*}
	In particular, for arbitrary coupling intensity $\kappa_1 > 0$, there exists $\beta > 1$ satisfying
	\begin{align}\label{D-8}
		\alpha \in (1,\beta) \quad \Rightarrow \quad \inf_{t \geq 0}\min_{ 1\leq i\neq j\leq N}\|x_i(t)-x_j(t)\| > 0.
	\end{align}
\end{theorem}

\begin{proof}
First assume that \eqref{D-6} holds. Suppose that a local well-posedness of our system holds on $t\in(0,t_*)$ for $t_*\in (0,\infty)$, then it follows from 
\[-\frac{d\|x_i-x_j\|}{dt}\leq \|v_i-v_j\|,\]
Theorem \ref{T3.1}, Theorem \ref{T3.2} and Remark \ref{R3.3} that

\begin{align*}
\|x_i(t)-x_j(t)\|&\geq \|x_i(0)-x_j(0)\|-\int_{0}^{t} \|v_i(s)-v_j(s)\|ds\geq |x^k_i(0)-x^k_j(0)|-\int_{0}^{t} D_V(s)ds\\
&\geq |x^k_i(0)-x^k_j(0)|-\int_{0}^{\infty} D_V(s)ds\geq |x^k_i(0)-x^k_j(0)|-\frac{T_M^\infty D_{V}(0)}{\kappa_1\mathcal{A}(v)(0){\phi(D_X^\infty)}}\\
&\geq \min_{ 1\leq i\neq j\leq N}|x^k_i(0)-x^k_j(0)|-\frac{T_M^\infty D_{V}(0)}{\kappa_1\mathcal{A}(v)(0){\phi(D_X^\infty)}}>0.
\end{align*}
Therefore, by the standard Cauchy-Lipschitz theory one can show that there exists a positive $\epsilon$ such that uniqueness and existence of the soultion \eqref{TCSUSS} on $(0,t_*+\epsilon)$ can be guaranteed. Hence, one can obtain the global well-posedness and strictly positivity of relative distance between inter-particles, which makes it possible to use the same strategies as in the proofs of Theorem \ref{T3.1} and Theorem \ref{T3.2}. Therefore we obtain the emergence dynamics of \eqref{TCSUSS}. \newline

Now suppose that $\alpha>1$ and \eqref{D-7} holds. This implies well-definedness of global solution and exponential decay of $D_V$; there exists positive constants $B,C$ satisfying
\[
	D_V(t) \leq Be^{-tC}.
\]
Therefore for any $i,j \in \{1,2,\cdots,N\}$, the limit $\lim_{t \to \infty}\|x_i(t)-x_j(t)\|$ always exists. Then we can define
	\[[l]:=\{i\in\{1,\cdots,N\}~|~\|x_l(t)-x_i(t)\|\to 0\quad\mbox{as}\quad t\to \infty \}.\]
Since collision does not happen in finite time from Lemma \ref{L3.1}, the proof of Lemma \ref{L3.1} extends globally, and therefore
\[
    \frac{dD_{V,[l]}}{dt} \leq -D_1\phi(D_{X,[l]})D_{V,[l]}+D_2,\quad \mbox{a.e.}~ t \in \bbr_{>0},
\]
for some positive constants $D_1,D_2>0$. Now define
\[
	\tilde{\mathcal{L}}(t) := D_1\int_{D_{X,[l]}(0)}^{D_{X,[l]}(t)}\phi(r)dr.
\]
Then for almost every $t > 0$, $|\tilde{\mathcal{L}}(t)|+D_{V,[l]}(t)$ have a linear or sub-linear growth;
\begin{align}\begin{aligned}\label{D-9}
\frac{d}{dt}|&\tilde{\mathcal{L}}(t)| + \frac{d}{dt}D_{V,[l]}(t)
\leq \left| \frac{d}{dt}\tilde{\mathcal{L}}(t) \right| + \frac{d}{dt}D_{V,[l]}(t)\\
&\leq \left| D_1\phi(D_{X,[l]}(t))\frac{d}{dt}D_{X,[l]}(t)  \right|
+\left( -D_1\phi(D_{X,[l]})D_{V,[l]}+D_2 \right) \leq D_2 < \infty, \quad \text{a.e.} ~ t \in \bbr_{>0}.
\end{aligned}\end{align}
	On the other hand, for sufficiently large $t \gg 1$, we obtain
	\begin{align}\begin{aligned}\label{D-10}
	\frac{\left|\tilde{\mathcal{L}}(t)\right|}{D_1}
	&=
	\left|\int_{D_{X,[l]}(0)}^{D_{X,[l]}(t)} \phi(s) ds\right|
	= \int^{D_{X,[l]}(0)}_{D_{X,[l]}(t)} \phi(s) ds \\
	&\geq \int^{D_{X,[l]}(0)}_{Be^{-tC}} \phi(s) ds
	= \frac{B^{1-\alpha}e^{tC(\alpha-1)}-\left(D_{X,[l]}(0)\right)^{1-\alpha}}{\alpha-1}.
	\end{aligned}\end{align}
Putting \eqref{D-9} and \eqref{D-10} altogether, there exists positive constants $D_3$ and $D_4$ satisfying
	\begin{align*}
		D_3e^{tC(\alpha-1)}-1 \leq \left| \tilde{\mathcal{L}}(t) \right| \leq D_4(1+t), \quad \text{a.e.} ~ t \in \bbr_{>0},
	\end{align*}
which yields a contradiction. Therefore we conclude $[l]$ is the empty set, which proves
	\begin{align*}
		\inf_{t \geq 0}\min_{ 1\leq i\neq j\leq N}\|x_i(t)-x_j(t)\| > 0.
	\end{align*} 
To prove \eqref{D-8}, suppose that $\kappa_1 > 0$ is given. Since
\[
	\lim_{\alpha \searrow 1}\int_{D_X(0)}^\infty x^{-\alpha} = +\infty,
\]
the condition \eqref{D-8} is always achieved when $\alpha>1$ is sufficiently close to 1. This proves \eqref{D-8}.

\end{proof}

\section{Conclusion} \label{sec:5}
In this paper, we have established the sufficient framework for collision avoidance in the dynamical system \eqref{TCSUSS} under both strong and weak singular kernels. To achieve this, we utilized the dissipative structures in terms of $L^\infty$-diameters and applied technical estimates and bootstrapping arguments to obtain global well-posedness, as well as to prove asymptotic flocking behavior and thermal equilibrium phenomenon under appropriate conditions on initial data and systemic parameters. Notably, we have also shown that collision can still occur in the two-particle model despite the singularity of communication, and we imposed suitable conditions on initial data to ensure collision avoidance and maintain strict spacing between agents. As for future work, the following problems remain to be addressed.
\begin{itemize}
	\item (Q1):~Can we enlarge and improve the sufficient framework for collision avoidance? \vspace{-0.2cm}
	\item (Q2):~Can we verify the non-collisional phenomenon on Riemannian manifolds as well?
	\vspace{0.2cm}
	\item(Q3):~Can we rigorously establish a local or global well posedness for the mesoscopic (kinetic) level?
\end{itemize} ~ \newline

\bibliographystyle{amsplain}

\end{document}